\documentclass[a4paper, twoside, notitlepage, 12pt]{amsart} 
\usepackage{amsmath,amsfonts,amssymb,amsthm,mathrsfs} 
\usepackage[matrix,arrow]{xy}

\usepackage{verbatim}
\usepackage[latin1]{inputenc}   
\usepackage{backref}

\title[Quot schemes in Grassmannians]{Quot schemes in Grassmannians}
\author{Roy Mikael Skjelnes}
\email{skjelnes@kth.se}

\thanks{}

\DeclareMathOperator{\im}{Im}
\DeclareMathOperator{\Spec}{Spec}

\newcommand{\Fitt}{\mathscr{F}itt}
\newcommand{\schFitt}{\mathrm{Fitt}}

\newcommand{\ufitt}[2]{{\mathscr{Fitt}}^{#2}_{#1}}
\newcommand{\ra}{\longrightarrow}
\newcommand{\Proj}[1]{\mathbb{P}({#1})}

\newcommand{\Sym}{\operatorname{S}}

\newcommand{\eqbeg}{\begin{equation}}
\newcommand{\eqend}{\end{equation}}

\newcommand{\Sgot}{\mathfrak{S}}

\newcommand{\m}{\mathfrak{m}}

\newcommand{\call}{\mathscr{L}}
\newcommand{\calV}{\mathscr{V}}
\newcommand{\calE}{\mathscr{E}}
\newcommand{\calK}{\mathscr{R}}
\newcommand{\calS}{\mathscr{S}}
\newcommand{\calO}{\mathscr{O}}

\newcommand{\calG}{\mathscr{G}}
\newcommand{\calF}{\mathscr{F}}
\newcommand{\calR}{\mathscr{R}}

\newcommand{\Snbimap}{\xymatrix@M=1pt{ {\Sgot_n\times U}
\ar@<.5ex>[r] \ar@<-.5ex>[r] & {U}}}
\newcommand{\bimap}{\xymatrix@M=1pt{ {R}
\ar@<.5ex>[r] \ar@<-.5ex>[r] & {X}}}
\newcommand{\pbbimap}{\xymatrix@M=1pt{ {R\times_{A}Y}
\ar@<.5ex>[r] \ar@<-.5ex>[r] & {X\times_AY}}}
\newcommand{\grpbimap}{\xymatrix@M=1pt{ {G\times X}
\ar@<.5ex>[r] \ar@<-.5ex>[r] & {X}}}
\newcommand{\pibimap}{\xymatrix@M=1pt{ {R}
\ar@<.5ex>[r]^-{\pi_1} \ar@<-.5ex>[r]_-{\pi_2} & {X}}}
\newcommand{\quotmap}{\xymatrix@M=1pt{ {R/G}
\ar@<.5ex>[r]^-{\pi_1} \ar@<-.5ex>[r]_-{\pi_2} & {U/G}}}

\newtheorem{thm}[subsection]{Theorem}
\newtheorem{lemma}[subsection]{Lemma}
\newtheorem{cor}[subsection]{Corollary}
\newtheorem{prop}[subsection]{Proposition}

\theoremstyle{definition}

\newtheorem{ex}[subsection]{Example}

\theoremstyle{remark}
\newtheorem{rem}[subsection]{Remark}

\numberwithin{equation}{subsection}

\begin{document}
\begin{abstract}  We describe the closed strata that defines the Quot schemes $\mathrm{Quot}^n_{\calO^p/{\mathbb P}^r/S}$ as closed subschemes in Grassmannians. Here $n$ is the constant polynomial and $\calO^{p}$ is the free, rank $p$ sheaf on projective $r$-space ${\mathbb P}^r$ over some affine base scheme $S$.
\end{abstract}
\maketitle
\section{Introduction}

The Quot functors ${\mathrm{Quot}}^P_{\calF/{\mathbb P}^r/S}$ parametrizing quotient sheaves of a fixed coherent sheaf $\calF$ on a projective scheme ${\mathbb P}^r \ra S$ that are flat over the base and having Hilbert polynomial $P(t)$ on the fibers, are arguably one of the most fundamental objects in algebraic geometry. The Grassmannian  and the Hilbert functors being  special cases. 

One of the important properties of the Quot functor is that it is a projective scheme \cite{FGAIV}. The original proof showing this property is based on an abstract embedding into a suitable Grassmannian: There exists a regularity bound, that is an integer $d$ depending on the polynomial $P(t)$, such that when twisting any point of the Quot functor with $d$, and thereafter pushing the sheaves down to the base, one obtains a morphism from the Quot functor to the Grassmannian of rank $P(d)$ quotients of the global sections of $\calF(d)$. Using a stratification argument Grothendieck shows that the morphism describes the Quot functor as locally closed, hence representable, subset of the Grassmannian. Then by showing that the Quot functor satisfies the evaluative criteria for properness, he concludes that the Quot scheme is proper, and in particular closed in the Grassmannian. A more detailed exposition was given by Mumford \cite{Mumfordlectures}, and variations can be found in e.g. \cite{FGAexplained}, \cite{C-F&Kapranov}, \cite{AltmanKleiman}.

Embeddings of the special case with  Hilbert schemes have been investigated thoroughly, and in particular the Gotzmann Persistence Theorem \cite{Gotzmann} was a major break through leading to a description of the defining equations of the Hilbert scheme inside the Grassmannian. We refer the reader to  Appendix C in \cite{IarrobinoKleiman} and the references therein, for more information about the defining equations of the Hilbert scheme. More recent work are found in \cite{Extensors} and \cite{HaimanSturmfels}.


We will in this article describe the defining strata for a class of Quot functors, namely with the sheaf $\calF=\calO^{\oplus p}$ being the free sheaf of rank $p$, and the polynomial $P(t)=n$ being the constant polynomial. Our description is as follows. Let $f\colon {\mathbb P}^r \ra S$ be the structure morphism. Fix an integer $d\geq n$,  and let $G$ denote the Grassmannian of locally free rank $n$ quotients of $f_*\calF(d)$, and let 
$$\xymatrix{ 0\ar[r] & \calR_d \ar[r] & f_*\calF(d) \ar[r] & \calE_d \ar[r] & 0}$$
denote the universal sequence on $G$. The $\calO _G$-module $\calR_d$ determines naturally a $\calO_G [X_0, \ldots, X_r]$ graded subsheaf $\calR$ of the free rank $p$ graded module; the cone generated by $\calR_d$. And then also a graded quotient module $\calE$. The components $\calE_m$ of the graded quotient sheaf $\calE$ are coherent sheaves, and in particular their Fitting ideals determine closed strata on $G$. Let $\mathrm{Fitt}_{n}(\calE')\subseteq G$ denote the closed subscheme defined by the $n$'th Fitting ideal of the coherent sheaf $\calE'$.  Our result is that the Quot scheme 
$$ \mathrm{Quot}^n_{{\calO}^{\oplus p}/{\mathbb P}^r/S}=\bigcap_{t\geq 1}\mathrm{Fitt}_{n-1}(\calE_{d+t}).$$
 Our result depends on, besides the set up and properties of Fitting ideals, the vector space dimension estimates obtained by Gasharov \cite{Gasharov}. It is worthwhile to point out that not only do we describe the closed strata defining the Quot scheme, but we also prove,  in an essentially new way,  the representability and projectivity of the Quot functors with constant polynomial. 


As the spaces considered are Noetherian only a finite number of the subschemes $\mathrm{Fitt}_{n-1}(\calE_m)$ are needed to describe the Quot scheme. We do not have a general bound or estimate for the number of intersection components  needed, but for some special cases we have a complete answer.  In the case with the projective line ${\mathbb P}^1$ we prove that only the first component is needed, that is
$$ \mathrm{Quot}^n_{\calO^{\oplus p}/{\mathbb P}^1/S}=\mathrm{Fitt}_{n-1}(\calE_{d+1}).$$
Similarly, we need only a single Fitting ideal stratum in the case of the Hilbert functor, that is when the fixed free sheaf has rank $p=1$.
 Thus 
$$\mathrm{Hilb}^n_{{\mathbb P}^r/S}=\mathrm{Fitt}_{n-1}(\calE_{d+1}), $$
 which turns out to be a direct consequence of Gotzmann's Persistence Theorem. The Hilbert schemes were known to be given by determinantial equations (\cite{HaimanSturmfels}, \cite{IarrobinoKleiman}), but our description in terms of Fitting ideals appears to be new - even though closely related to existing descriptions.

We show furthermore that the single Fitting ideal condition is set theoretically enough to describe the Quot schemes; that is for the underlying topological spaces we have the equality 
$$\mid \mathrm{Quot}^n_{{\calO}^{\oplus p}/\mathbb{P}^r/S}\mid =\mid \mathrm{Fitt}_{n-1}(\calE_{d+1})\mid.$$
Interestingly enough, we know of no examples where the actual scheme structures are different. In fact, showing that the nilpotent elements needed to describe the Quot scheme from $\mathrm{Fitt}_{n-1}(\calE_{d+1})$ are not only nilpotent, but actually zero is equivalent with a  natural generalization of the Gotzmann Persistence Theorem to modules.

We end by computing some examples that we find interesting. First we redo some known examples of the Hilbert scheme of points and put these in the framework of the presentation given here. Thereafter we compute some new examples of the Quot scheme of the projective line.

\subsection*{Acknowledgments} The discussions that I had together with Greg Smith (2010/2011) and together with Runar Ile (2011/2012) during their respective sabbaticals, preceded the discoveries presented in this article.

\section{Fitting ideals}

We let $S=k[X_0, \ldots, X_r]$ denote the standard graded polynomial ring in the variables $X_0, \ldots, X_r$ over a field $k$. If $F$ is a $S$-module, we mean a ${\mathbb Z}$-graded $S$-module $F=\bigoplus _{m\in {\mathbb Z}}F_m$. 

\subsection{Macaulay representation} For any non-negative integer $n$ we have for any positive integer $d$, the $d$-th Macaulay representation
$$ n =\binom{m_d}{d}+\binom{m_{d-1}}{d-1}+\cdots +\binom {m_1}{1},$$
with $m_d> m_{d-1}> \cdots > m_1\geq 0$. The convention is that $\binom{m}{d}=0$ when $m< d$. We define
$$ n^{\langle d\rangle} =\binom{m_d+1}{d+1}+\binom{m_{d-1}+1}{d}+\cdots +\binom {m_1+1}{2}.$$

\begin{prop}\label{Gasharov}  Let $F=\bigoplus_{i=1}^pS$ be the free $S=k[X_0,\ldots, X_r]$-module of rank $p$. Let $R\subseteq F$  be a graded $S$-submodule, and let $E=\bigoplus_{m\geq 0}{E_m}$ denote the graded quotient. We fix a degree $d\geq 1$, and let $n=\mathrm{dim}_k(E_d)$.
\begin{enumerate}
\item We have that the vector space dimension $\mathrm{dim}E_{d+1}\leq n^{\langle d\rangle}$.
\item Assume that $\mathrm{dim}(E_{d+1})=n^{\langle d\rangle}$, and that $R$ is generated in degree $d$,  then we have the equality $\mathrm{dim}(E_{d+2})=(\mathrm{dim}E_{d+1})^{\langle d+1 \rangle}$.
\end{enumerate}
\end{prop}

\begin{proof} This is a special case of \cite[Theorem 4.2]{Gasharov}.
\end{proof}

\begin{rem} The first assertion above can be proven by reducing to the case with the initial module, and thereafter applying Macaulays theorem  \cite{Macaulay}, \cite[Theorem 4.2.10]{BH}.  We will in this note use both assertions in the Proposition, but not the slightly more general versions that appear in \cite{Gasharov}.
\end{rem}

\begin{rem}\label{constant expansion} We will be applying the above mentioned result for the constant polynomial case: Consider the situation with  $n\leq d$. We then have that
$$ n =\binom{d}{d}+\binom{d-1}{d-1}+\cdots +\binom{d-n+1}{d-n+1}.$$
From where it  follows that $n^{\langle d +s\rangle}=n$, for $s\geq 0$, as well.
\end{rem}

\subsection{Fitting ideals} Let $E$ be a finitely generated $A$-module. The $n$'th  Fitting ideal we denote by $\Fitt_n(E)\subseteq A$ \cite{Bourbaki4-7}, \cite{Northcott}. If we have a presentation of the $A$-module $E$ as the co-kernel of a map of free modules
\begin{equation}\label{fitting presentation}
\xymatrix{ \bigoplus_{\alpha}A \ar[r]^{\psi} & \bigoplus_{i=1}^NA\ar[r] & E\ar[r] & 0,}
\end{equation}
then the $(N-n)$-minors of $\psi$ generate the $n$'th Fitting ideal $\Fitt_n(E)$. We recall that $\Fitt_{-1}(E)=0$, and more generally that $\Fitt_n(E)=0$ if the matrix representing $\psi$ does not have $N-n$-minors.  Moreover, the determinant of the $0\times 0$ matrix is one, so when $n\geq N$ then $\Fitt_n(E)=A$.

\begin{prop}\label{Fitting} Let $(A, \m)$ be a local ring, and $E$ a finitely generated $A$-module. Then $E$ is free of rank $n$ if and only if the two following conditions are satisfied
\begin{enumerate}
\item The vector space dimension $\operatorname{dim}(E\bigotimes_AA/\m)=n$.
\item The Fitting ideal $\Fitt_{n-1}(E)=0$.
\end{enumerate}
\end{prop}

\begin{proof} If $E$ is free, then we can let $\psi $ be the zero-map in \ref{fitting presentation} and clearly the conditions hold.

We will prove the converse. By assumption the dimension of the vector space $E\bigotimes_AA/\m$ is $n$. Nakayamas Lemma then guarantees that a lift of a basis  will generate the $A$-module $E$.  We can thus find an $A$-linear map of free modules
$ \xymatrix{\psi \colon \bigoplus_{\alpha }A\ar[r] &\bigoplus_{i=1}^n A}$
whose co-kernel is $E$. The Fitting ideals are independent of the presentation, so the $n-(n-1)$-minors of the map $\psi$ generate the Fitting ideal $\Fitt_{n-1}(E)$. By assumption this ideal is the zero ideal, hence $\psi$ is the zero map and we get that its co-kernel $E=\bigoplus_{i=1}^nA$.
\end{proof}

\subsection{Graded rings} Let $A$ be a commutative unital ring, and let $V$ be an $A$-module. The symmetric tensor algebra $\Sym_A(V)$ is naturally graded
\begin{equation}\label{graded ring}
 \xymatrix{\Sym_A(V)=\bigoplus _{m\geq 0}\Sym_m(V)}.
\end{equation}
We have that $V=\Sym_1(V)$ is the degree one component. 


\begin{cor}\label{vector space and free} Let $(A,\m )$ be a local ring, $V$ a finitely generated $A$-module,  and set $F=\bigoplus_{i=1}^p\Sym_A(V)$. Let $R\subseteq F$ be a submodule generated in degree $d$. Assume that the vector space $F_d/R_d\bigotimes_AA/\m$ has dimension $n\leq d$. Then, for any $m\geq d$ we have that $E_m=F_m/R_m$ is free of rank $n$ if and only if $\Fitt_{n-1}(E_m)=0$.
\end{cor}
\begin{proof} Assume that $\mathrm{Fitt}_{n-1}(E_m)=0$. By the proposition we need to verify that the vector space dimension of $E_m\bigotimes_AA/\m$ is $n$. Let $\psi \colon G' \ra G=F_m\bigotimes_AA/\m$ be a map of $A/\m $-vector spaces whose cokernel is $E_m\bigotimes_AA/\m$. Since $\mathrm{Fitt}_{n-1}(E_m)=0$ we have that  all $\mathrm{rk} (G)-n+1$ minors of $\psi$ are zero. In other words the image $\im{\psi}$ has dimension strictly less than $\mathrm{rk}(G) -n+1$. Thus,
$$\operatorname{dim}(\im (\psi))\leq \mathrm{rk}(G) -n.$$
That means that the co-kernel has dimension $\dim (E_m\bigotimes_AA/\m)\geq n$. We have that $S_A(V)\bigotimes_AA/\m=A/\m [X_0, \ldots, X_r]$ where $r+1$ is the vector space dimension of $V\bigotimes_AA/\m$. By Proposition \ref{Gasharov} (1) we have the inequality the other way around, and we obtain that the vector space dimension of $E_m\bigotimes_AA/\m$ is exactly $n$. 
The reverse implication follows from the Proposition.
\end{proof}

\subsection{Union of Fitting ideals} Let $E=\{E_m\}_{m\geq 0}$ be a collection of finitely generated $A$-modules. For fixed integer $n$, and non-negative integers $d$ and $s$ we define the ideal $\ufitt{n}{d,s}(E)\subseteq A$ to be the ideal
$$ \ufitt{n}{d,s}(E):=(\Fitt_n(E_{d}), \Fitt_n(E_{d+1}), \ldots , \Fitt_n(E_{d+s})).$$
We then have, for fixed integers $d$ and $n$, the inclusion of ideals
\begin{equation}\label{union of fittings}
\ufitt{n}{d,0}(E)\subseteq \ufitt{n}{d,1}(E) \subseteq \ufitt{n}{d,2}(E)\subseteq \cdots .
\end{equation}
Their union we denote by 
\begin{equation}\label{ideal infty}
 \lim_{s\to \infty}\ufitt{n}{d,s}(E)= \ufitt{n}{d, \infty}(E).
\end{equation}
If $E=\bigoplus_{m\geq 0}E_m$ is a direct sum of $A$-modules, then we write $\ufitt{n}{d,s}(E)$ for the ideal that we assign to the collection $\{E_m\}_{m\geq 0}$.


\begin{cor}\label{factorization} Let $R\subseteq F=\bigoplus_{i=1}^p\Sym_A(V)$ be a submodule generated in degree $d$, and let $E=F/R$. Let $f\colon A \ra B$ be a homomorphism of rings, and let $s\geq 0$ be an integer.  Assume that for every prime ideal $\m $ in $A$, the vector space $E_d\bigotimes_A\kappa(\m)$ is of dimension $n$, where $n\leq d$. Then the morphism $f$ factorizes via $A\ra A/\ufitt{n-1}{d,s}(E_d)$ if and only if $E_m\bigotimes_AB$ are projective $B$-modules of rank $n$, for $d\leq m\leq s+d$.
\end{cor}

\begin{proof} If $E_m\bigotimes_AB$ is projective $B$-module of rank $n$, then in particular $\Fitt_{n-1}(E_m\bigotimes_AB)=\Fitt_{n-1}(E_m)B$ is zero, and the factorization follows. The converse follows from Corollary \ref{vector space and free}.
\end{proof}

\section{The Quot functor and regularity of sheaves}

In this section we will recall some known facts about Castelnuovo-Mumford regularity. We have included these results for completeness, and we do not claim any originality in our presentation.

\subsection{Castelnuovo-Mumford regularity}  Let $f\colon X\ra S$ be a projective morphism, and let ${\calO}_X(1)$ be a very ample line bundle. For any sheaf $\calF$ on $X$, we let $\calF(m)=\calF\bigotimes_X\calO_X(m)$. If $s\in S$ is a point, then $\calF\bigotimes_S\kappa (s)$ denotes the restriction of the sheaf $\calF$ to the fiber $f^{-1}(s)$. A coherent sheaf $\calF$ on a projective scheme $X\ra S$ is fiber wise $m_0$-regular if for every point $s\in S$ the sheaf $\calF\bigotimes_S\kappa (s)$ is $m_0$-regular, that is
$$\xymatrix{ H^i(f^{-1}(s), \calF(m_0-i)\bigotimes_S\kappa(s))=0 \quad \text{for}\quad i\geq 1.}$$ 
A sheaf $\calF$ that is $m_0$-regular is also $m$-regular for any $m\geq m_0$.

\begin{prop}\label{d regular}Let $X\ra \Spec(k)$ be a projective morphism, with  $k$ a field,  and let ${\calO}_X(1)$ be a very ample line bundle. Let $\calF$ be a coherent and $m_0$-regular sheaf on $X$. Let $\calF\ra \calE$ be a coherent quotient sheaf, with finite support, and let $\operatorname{dim}H^0(X, \calE)=n$. Then we have the following.
\begin{enumerate}
\item The sheaf $\calE$ is $m_0$-regular.
\item The kernel $\mathrm{ker}(\calF\ra \calE)$ is $m_0+n$-regular.
\end{enumerate}
\end{prop}

\begin{proof} We may assume that  the base field $k$ has infinite many elements, and we may assume that $X=\mathbb{P}^r$ is the projective $r$-space. If $r=0$ there is nothing to prove, and we assume therefore that $r>0$. We can then find a hyperplane $H\subset X$ avoiding the finite set $\mathrm{Supp}(\calE)$. Let $\calK$ denote the kernel of $\calF\ra \calE$.  We tensor the exact sequence 
$$ 0 \ra \calK(-1) \ra \calK \ra \calK_{|H} \ra 0,$$
with $\calO (m+1-i)$. After taking the long exact sequence in co-homology, and using the fact that  $\calK_{|H}=\calF_{|H}$ we obtain  that $H^i(\calK(m-i))=H^i(\calK(m+1-i))$ for all $m\geq m_0$, and all $i\geq 2$, and where we write $H^i(\calF)$ when we mean $H^i(X,\calF)$.  Hence these groups are zero, that is $H^i(\calK(m-i))=0$, for $m\geq m_0$ and $i\geq 2$. 

From the short exact sequence $0\to \calK \to \calF \to \calE\to 0$, we obtain after after tensoring with $\calO (m-i)$ the exact sequence
$$H^i(\calF (m-i)) \ra H^i(\calE (m-i)) \ra H^{i+1}(\calK (m+1-i-1)).$$
For $i\geq 1$, and $m\geq m_0$ we have that the left hand term is zero by assumption, and that the right hand side is zero by the observation above. This proves the first statement about $m_0$-regularity of $\calE$.

For $m\geq m_0-1$ we have, by the assumption on the regularity of $\calF$,  an exact sequence
$$ H^0(\calE(m)) \ra H^1(\calK(m+1-1)) \ra 0.$$
Hence $H^1(\calK(m))$ is bounded by the dimension of $H^0(\calE(m))$. Since $\calE$ has finite support, we have that $H^0(\calE(m))=H^0(\calE)=n$. It is well-known that  the dimension of $H^1(\calF(m+1))$ is strictly less than the dimension of $H^1(\calF(m))$, provided that the latter group is non-zero.  We then get that 
$$\operatorname{dim}H^1(\calK (m+1+n-1))=0,$$
when $m\geq m_0-1$. Combined with the vanishing observed above, we obtain that $\calK$ is $m_0+n$-regular.
\end{proof}

\begin{rem} The proof of the Proposition \ref{d regular} is a simplified version of the more general statement with non-constant polynomials (see Mumford  \cite{Mumfordlectures}). Our simplified version presented here is not only for completeness and readability, but we also wanted to state the, well-known, explicit regularity bound for quotients having constant Hilbert polynomial. 
\end{rem}

\begin{cor}\label{regular families} Let $f\colon X \ra S$ be a projective morphism of Noetherian schemes, and let ${\calO}_X(1)$ be a very ample line bundle. Let $\calE$ be a coherent sheaf on $X$ that is flat over $S$, and fiber wise $m_0$ regular. Then we have, for each $m\geq m_0$ that
\begin{enumerate}
\item The canonical map $f^*f_*\calE(m) \ra \calE(m)$ is surjective.
\item The canonical map $f_*\calE(m) \bigotimes_Sf_*{\calO}(1) \ra f_*\calE(m+1)$ is surjective.
\item The sheaves $R^if_*\calE(m)=0$ for $i\geq 1$, and the sheaf $f_*\calE(m)$ is locally free.
\item If we have a Cartesian diagram of (Noetherian) schemes
$$ \xymatrix{X' \ar[r]^{\varphi'} \ar[d]_{f'} & X\ar[d]^{f} \\S'\ar[r]^{\varphi}& S, }$$
then the canonical map $\varphi^*f_*\calE(m) \ra f'_{*}{\varphi'}^{*}\calE(m)$ is an isomorphism.
\end{enumerate}
\end{cor}
\begin{proof} The canonical maps in (1) and (2) are  surjective if and only if surjective when restricted to each fiber $f^{-1}(s)$, for any $s\in S$. By assumption the sheaves $\calE(m_0)$ restricted to a fiber, are $m_0$-regular, where the statements hold \cite[Proposition p. 99]{Redbook}. To prove statements (3), we use that the sheaves $R^if_*\calE(m)$ are coherent so it suffices to show that $R^if_*\calE(m)\bigotimes_S\kappa (s)$ are zero, for any point $s\in S$. The canonical map
$$\xymatrix{ R^if_*\calE(m)\bigotimes_S \kappa (s) \ra H^i(f^{-1}(s), \calE(m)\bigotimes_S \kappa (s))}$$
is clearly surjective for $i\geq 1$ and $m\geq m_0$. Since $\calE$ was assumed flat, the canonical map is an isomorphism, and the statements (3)  follows from the co-homological base change result \cite[Theorem 12.11]{H52}. The last Statement (4) follows by statements (3) proven combined with   \cite[6.9.9.2]{EGA3II}.
\end{proof}

\subsection{Quot functor} \cite{FGAIV} Let $f\colon X\ra S$ be a morphism of schemes, and ${\calF}$ a quasi-coherent sheaf on $X$. We let
$$ \mathrm{Quot}^n({\calF}/X/S) $$
denote the collection of quasi-coherent ${\calO}_X$-module quotients $\calF \ra \calE$, where $\calE$ is flat, finite and of rank $n$ over the base scheme $S$. Two quotients $\calF\ra \calE $ and $\calF\ra \calE'$ are considered as equal in $\mathrm{Quot}^n(\calF /X/S)$ if their respective kernels are equal as subsheaves of $\calF$.  For any morphism of schemes $\varphi\colon S'\ra S$, let $\varphi'\colon X' =X\times_SS' \ra X$, and let ${\calF}'=\varphi'^*\calF$.  The Quot functor $\underline{\mathrm{Quot}}^n_{\calF/X/S}$  is the contra-variant functor that to every $S$-scheme $S'$ assigns the set $\mathrm{Quot}^n({\calF'}/X'/S') $.
 
\subsection{Grassmannians} For any quasi-coherent sheaf $\calG$ on a scheme $S$, we let $\mathrm{Grass}^n(\calG)$ denote the Grassmannian parametrizing locally free rank $n$ quotients of $\calG$. If $\varphi \colon S' \ra S$ is a morphism of schemes, then the $S'$-valued points of the Grassmannian $\mathrm{Grass}^n(\calG)$ is the set of quasi-coherent ${\calO}_{S'}$-module quotients $\varphi^*\calG \ra \calE$ that are locally free of rank $n$. Two quotients are identified if their kernels coincide. It is an elementary fact that the Grassmannian is a scheme \cite[Proposition 9.7.7]{EGA1}.

\begin{cor}\label{vanishing of direct images} Let $f\colon X \ra S $ be a projective morphism of Noetherian schemes, and let ${\calO}_X(1)$ be a very ample line bundle. Let $\calF$ be a coherent sheaf on $X$ that is flat and fiber wise $m_0$-regular over $S$. Then we have, for any $m\geq m_0$,  a natural map 
\begin{equation}\label{iota}
 \iota \colon \mathrm{Quot}^n_{{\calF}/X/S} \ra \mathrm{Grass}^n(f_*{\calF}(m+n)).
\end{equation}
\end{cor}
\begin{proof} Let $\varphi \colon S'\ra S$ be a morphism of schemes,  and let $\calK$ denote the kernel $\mathrm{ker}({\varphi '}^{*}\calF=\calF' \ra \calE)$, where $\varphi'\colon X\times_SS'=X' \ra X$ is the induced projection map, and where $\calE$ is flat and with constant Hilbert polynomial $n$. Since $\calE$ is flat, we have an exact sequence
$$ \xymatrix{0 \ar[r] & \calK \bigotimes_{S'}\kappa (s) \ar[r] & \calF'\bigotimes_{S'} \kappa (s) \ar[r] & \calE \bigotimes_{S'}\kappa (s) \ar[r] & 0},$$
when restricted to a fiber over a point $s\in S'$. By the proposition we have that the sheaf $\calK\bigotimes_{S'}\kappa (s)$ is $m_0+n$-regular. Since $\calF$ was assumed flat, so is $\calK$, and it follows by the Corollary above that $R^i{f'}_{*}\calK(m+n)=0$ for $i\geq 1$, and $m\geq m_0$, where $f' \colon X' \ra T$. Thus we have an exact sequence
$$\xymatrix{ 0\ar[r] & {f'}_{*}\calK(m+n) \ar[r] & {f'}_{*}\calF' (m+n) \ar[r] & {f'}_{*}\calE(m+n) \ar[r] & 0}$$
of coherent sheaves on $S'$. Hence we have that ${f'}_{*}\calE(m+n)$ is a quotient of ${f'}_{*}{\varphi '}^*\calF(m+n)$, and it is furthermore locally free by the Corollary \ref{regular families}, which also gives the functoriality of the constructed map.
\end{proof}

\begin{rem} The natural map \ref{iota} is the one main ingredient in proving the representability of the Quot functor. Grothendieck \cite{FGAIV} then proceeds with showing that the map is locally a closed immersion, which then implies that the Quot functor is representable with a quasi-projective schemes. Thereafter he establishes that the Quot functor satisfies the valuative criteria for properness which proves that the Quot functor is represented by a closed subscheme of the Grassmannian. Other existing proofs  (e.g. \cite{Mumfordlectures}, \cite{StrommeBanach}, \cite{FGAexplained}, \cite{HaimanSturmfels}, \cite{AltmanKleiman}) more or less follow this strategy. A different approach with the Hilbert functor being a section of a representable  Homogeneous Hilbert functor is found in \cite{HomogeneousHilb}.  
\end{rem}

\section{The sheaf cone and Fitting ideals}

\subsection{Projective space} If $V$ is an $A$-module, we have the associated projective space $\Proj{V}=\operatorname{Proj}(\Sym_A(V))$. The scheme $\Proj{V}\ra \Spec(A)$ comes equipped with an universal line bundle $\calO(1)$, which is very ample when $V$ is finitely generated. From a graded $\Sym_A(V)$-module $E$ we obtain in a natural way a sheaf $\calE$ on $\Proj{V}$. 


\begin{lemma}\label{global sections} Let $A$ be a Noetherian ring, $V$ a finitely generated and projective $A$-module. Let $F=\bigoplus_{i=1}^p\Sym_A(V)$, and let $\calF$ denote the associated free sheaf on $\Proj{V}$.
\begin{enumerate}
\item For any integer $m$ we have that $F_m=H^0(\Proj{V},\calF(m))$.
\item Let $ \xymatrix{ 0\ar[r] & \calR \ar[r]& \calF \ar[r] &  \calE \ar[r] & 0}$ be a short exact sequence of coherent sheaves on $\Proj{V}$. Assume that $\calE$ is flat, finite and of rank $n$ over $\Spec(A)$. Let $d\geq n$, and let $R\subseteq F$ denote the graded $\Sym_A(V)$-submodule generated by 
$H^0(\Proj{V},\calR (d))$. Then we have, for each $m\geq d$ that 
$$ \xymatrix{F_m/R_m=H^0(\Proj{V}, \calE(m))}.$$ 
\end{enumerate} 
\end{lemma}

\begin{proof} We have a natural map $F_m \ra H^0(\mathbb{P}(V), \calF(m))$. The first result then follows from \cite[Proposition 5.13]{H52}. We will verify the second statement. We have that $\calR$ is flat, and from Proposition \ref{d regular} we have that $\calR$ is $d$-regular, for any $d\geq n$,  on each fiber. By Proposition \ref{regular families} (3) we then have the exact sequence for $m\geq d$, 
$$ \xymatrix{ 0\ra H^0(\Proj{V},\calR(m)) \ra  F_m \ra H^0(\Proj{V},\calE(m)) \ra 0.}$$
By Proposition \ref{regular families} (2) the graded module $\bigoplus_{m\geq d} H^0(\Proj{V}, \calR(m))$ is generated by $H^0(\Proj{V},\calR(d))$.
\end{proof}

\subsection{} With $V$ an $A$-module, we let $f_A\colon \Proj{V} \ra \Spec(A)$ denote the structure sheaf. If $\varphi \colon \Spec(B) \ra \Spec(A)$ is a morphism of  schemes, we get the Cartesian diagram
$$ \xymatrix{ \Proj{V_B} \ar[d]_{f_B} \ar[r]^{\varphi'} & \Proj{V} \ar[d]^{f_A} \\\Spec(B) \ar[r]^{\varphi} & \Spec(A), }$$
where $V_B=V\bigotimes_AB$.
\begin{prop}\label{sheaves on PV} Let $V$ be a projective $A$-module of finite rank, with $A$ a Noetherian ring. Let $R\subseteq F=\bigoplus_{i=1}^m\Sym_A(V)$ be a submodule generated in degree $d$.  Let $\calE$ denote the sheaf on $\Proj{V}$ associated to the $\Sym_A(V)$-module $E=F/R$.
Assume that the $A$-module $E_d$ is projective of finite rank $n$, and that this rank $n\leq d$. Then we have that a morphism of schemes $\varphi \colon \Spec(B) \ra \Spec(A)$ factors via $\Spec(A/\ufitt{n-1}{d, \infty}(E))$ if and only if the two following conditions hold
\begin{enumerate}
\item The sheaf ${\varphi '}^*\calE$ is flat, finite of rank $n$ over $\Spec(B)$.
\item The $B$-module of global sections of ${\varphi '}^{*}\calE(m)$ is
$\xymatrix{ E_m\bigotimes_AB}$, as quotients of $F_m\bigotimes_AB$, for each $m\geq d$.
\end{enumerate}
\end{prop} 

\begin{proof}  If a  morphism of schemes $\varphi \colon \Spec(B) \ra \Spec(A)$ factors via $\Spec(A/\ufitt{n-1}{d,\infty}(E))$, then we have by Corollary \ref{factorization} that the $B$-modules $E_m\bigotimes_AB$ are projective of rank $n$, for all $m\geq d$. Hence the sheaf ${\varphi '}^*\calE$, which we obtain from the $\Sym_B(V_B)$-module $E\bigotimes_AB$, is flat, finite of  rank $n$ over the base. The  second assertion follows by Lemma \ref{global sections} (2). 

Conversely, let $\varphi \colon \Spec(B) \ra \Spec(A)$ be a morphism. By assumption the sheaf ${\varphi '}^{*}\calE=\calE'$ on $\Proj{V_B}$ is flat, finite and of rank $n$ over $B$. By Corollary \ref{vanishing of direct images} the sheaves $f_{B *}\calE'(m)$ are locally free for every $m\geq n$. By assumption the $B$-modules $E_m\bigotimes_AB$ are the global sections of $\calE'(m)$. Hence the $B$-modules $E_m\bigotimes_AB$ are projective of rank $n$, and the factorization of $\varphi \colon \Spec(B) \ra \Spec(A)$ follows by Corollary \ref{factorization}
\end{proof}

\subsection{Fitting strata}  For each quasi-coherent sheaf $\calE'$ of finite type on a scheme $S$, we let $\schFitt _n(\calE')\subseteq S$ denote the closed subscheme defined by the $n$'th Fitting ideal sheaf $\Fitt_n(\calE')$.



\subsection{The sheaf cone} Let $\calV$ be a locally free sheaf of finite rank on a Noetherian scheme $S$, and assume that $f\colon \Proj{\calV}\ra S$ comes with a very ample invertible sheaf $\calO(1)$. Let $\calF$ be the free rank $p$ sheaf on $\Proj{\calV}$, and let $d\geq 0$ be an integer. Assume that we have, on $S$,  a short exact sequence of coherent sheaves 
\begin{equation}\label{up n down}
\xymatrix{ 0 \ar[r] & \calR_d \ar[r]& f_*\calF(d) \ar[r] & \calE_d \ar[r] & 0.} 
\end{equation}
The canonical map $f_*{\calF}(d)\bigotimes_Sf_*{\calO}(1)\ra f_*{\calF}(d+1)$ is surjective (Corollary \ref{regular families}). We let $\calR_{d+1}\subseteq f_*{\calF}(d+1)$ denote the image of the induced map
$$\xymatrix{\calR_d\bigotimes_Sf_*{\calO}(1) \ar[r] & f_*{\calF}(d+1).}$$
By Lemma \ref{global sections} we have that $\calV=f_*\calO(1)$, and we have on $S$, an exact sequence of $\calS ym_{\calO_S}(\calV)$-graded quasi-coherent sheaves  
$$\xymatrix{ 0 \ar[r] & \calR=\bigoplus_{m\geq d}\calR_m  \ar[r]& \bigoplus_{m\geq d}f_*\calF(m) \ar[r] & \calE \ar[r] & 0}.$$
The sheaf $\calE=\bigoplus_{m\geq d}\calE_m$ is by definition the graded quotient of the two other sheaves. We refer to $\calE$ as the sheaf cone associated to the short exact sequence \ref{up n down}.



\begin{thm}\label{main result} Let $V$ be a projective and finitely generated module on a Noetherian ring $A$. Let $\calF$ denote the free sheaf of rank $p$ on $f\colon \Proj{V}\ra \Spec(A)=S$. Fix two integers $n\leq d$, and let $\operatorname{G}\ra S$ denote the Grassmann scheme parametrizing locally free rank $n$ quotients of  $f_*\calF(d)$. Let
$$\xymatrix{ 0\ar[r]& \calR_d \ar[r] & f_*{\calF}(d) \ar[r] & \calE_d \ar[r] & 0}$$
denote the universal short exact sequence on the Grassmannian $\operatorname{G}$, and let $\calE$ denote the associated sheaf cone. Then we have that 
$$\xymatrix{ \mathrm{Quot}^n_{\calF/ \Proj{V}/S} =\bigcap_{s\ge 1}\mathrm{Fitt}_{n-1}(\calE_{d+s}).}$$
\end{thm} 

\begin{proof}  Let $\varphi \colon T\ra S$ be a morphism of schemes. We may assume $T=\Spec(B)$ is affine. Then $\varphi^*f_*\calF(d)$ corresponds to $F_d$, where $F=\bigoplus_{i=1}^p\Sym_B(V\bigotimes_AB)$. Let $\varphi^*f_*\calF(d)\ra \calE$ be a $T$-valued point of the Grassmann scheme, and let $\calR_d$ denote its kernel. We then obtain a graded $\Sym_B(V_B)$-submodule $R\subseteq F$ generated in degree $d$ by the $B$-module $R_d$ that corresponds to the kernel sheaf $\calR_d$. Let $E=F/R$. By Proposition (\ref{sheaves on PV}) we have that the graded module $E$ gives a sheaf on $\Proj{V_B}$ whose restriction to $\Proj{V}\times_A\Spec(B/\ufitt{n-1}{d, \infty}(E))$ is flat, finite and of rank $n$ over $\Spec(B/\ufitt{n-1}{d, \infty}(E))$. This shows that we have a morphism 
$$ \xymatrix{ s\colon  \bigcap_{s\geq1} \mathrm{Fitt}_{n-1}(\calE_{d+s}) \ra \mathrm{Quot}^n_{\calF/\Proj{\calV}/S}}.$$

If we instead were given a $T=\Spec(B)$-valued point of the Quot scheme $\mathrm{Quot}^n_{\calF/\Proj{\calV}/S}$, then we have the following sequence
$$ \xymatrix{ 0\ar[r] & \calR \ar[r] & \calF_T=\bigoplus_{i=1}^p\calO_T \ar[r] & \calE \ar[r] & 0,}$$
where $\calE$ is flat, finite of rank $n$ over $T$. Let $R\subseteq F=\bigoplus_{i=1}^p\Sym_B(V\bigotimes_AB)$ denote the $\Sym_B(V_B)$-module generated in degree $d$ by the $B$-module $R_d=H^0(\Proj{V_B}, \calR(d))$.  We have by Lemma (\ref{global sections}) that the sheaf associated to the graded module $E=F/R$ is $\calE$.  This implies, by Proposition \ref{sheaves on PV} that the natural map \ref{iota} factorizes via $\cap_{s \ge 1}\mathrm{Fitt}_{n-1}(\calE_{d+s})\subseteq \operatorname{G}$, and it also implies that the map $\iota$ is the inverse of $s$.
\end{proof}

\begin{cor}\label{topos} With the assumptions as in the Theorem. We have an equality of the underlying topological spaces
$$ \mid \mathrm{Quot}^n_{\calF/\Proj{V}/S}\mid =\mid \mathrm{Fitt}_{n-1}(\calE_{d+1})\mid.$$
That is, the scheme $\mathrm{Fitt}_{n-1}(\calE_{d+1})$ cuts out set theoretically the Quot scheme $\mathrm{Quot}^n_{\calF/\Proj{V}/S}$.
\end{cor}

\begin{proof} Let $\Spec(B)$ be an open subset of $G$. Let $R\subseteq F=\bigoplus_{i=1}^p\Sym_B(V_B)$ be the graded submodule generated in degree $d$ by the $B$-module $R_d$ that corresponds to the restriction of the sheaf $\calR_d$ to $\Spec(B)$. By the Theorem we need to show that the elements defining the ideals $\Fitt_{n-1}(E_{m})$ (with $m\geq d+2$) are nilpotent in the ring
$$ B/\Fitt_{n-1}(E_{d+1}),$$
where $E=F/R$. Or equivalently, for any morphism $f\colon B \ra L$ to a field $L$, we have that $\Fitt_{n-1}(E_{m}\bigotimes_BL)=0$ for $m\geq d+2$. This again is equivalent with the vector space $E_m\bigotimes_BL$ having dimension $n$, which follows from the Proposition \ref{Gasharov} (2).
\end{proof}


\begin{cor}\label{GG} With the assumptions as in the Theorem. We have that
$$ \mathrm{Hilb}^n_{\Proj{V}/S}=\mathrm{Fitt}_{n-1}(\calE_{d+1}).$$
\end{cor}

\begin{proof} Using the notation as in the proof of previous Corollary, with $\Spec(B')$ an open subset of the Grassmannian. Let $B=B'/\ufitt{n-1}{1}(E_d)$. We have by construction that the $B$-modules that $E_m\bigotimes_AB$, where $m=d, d+1$, are projective of rank $n$. It then follows by Gotzmanns persistence Theorem \cite{Gotzmann} that $E_m\bigotimes_{B'}B$ are projective for all $m\geq d$. In particular we get that $\Fitt_{n-1}(E_m\bigotimes_{B'}B)=0$, and the result follows from the Theorem. 
\end{proof}

\begin{cor} Let $X\ra \Spec(A)=S$ be a projective morphism, and let $\calO_X(1)$ be a very ample line bundle. Let $\calF$ be a coherent sheaf on $X$. Then the functor $\mathrm{Quot}^n_{\calF/X/S}$ is representable by a projective scheme.
\end{cor}

\begin{proof} The reduction to the ${\mathbb P}^r$ case is well-known, see (\cite{FGAIV}): Let $i\colon X\subseteq {\mathbb P}^r_A$ be the closed immersion given by ${\calO}_X(1)$. Then $i_*\calF$ is a coherent sheaf on ${\mathbb P}^r_A$, and it follows that the natural map
$$ \mathrm{Quot}^n_{\calF/X/S} \ra \mathrm{Quot}^n_{i_*\calF/{\mathbb P}^r/S}$$
is a representable closed map. Hence we may assume that $X={\mathbb P}^r_A$. There exists an integer $N$ such that ${\calF}(N)$ is generated by its global sections, thus we have a surjection $\bigoplus_{i=1}^r\calO \ra {\calF}(N)$.  It now follows that the natural map
$$ \mathrm{Quot}^n_{\calF(N)/{\mathbb P}^r/S}\ra \mathrm{Quot}^n_{\calO^{\oplus p}/{\mathbb P}^r/S}$$
is a representable closed map. Hence it follows from the Theorem that $\mathrm{Quot}^n_{\calF(N)/{\mathbb P}^r/S}$ is representable. As $\mathrm{Quot}^n_{\calF(N)/{\mathbb P}^r/S}=\mathrm{Quot}^n_{\calF/{\mathbb P}^r/S}$ we are done.
\end{proof}

\section{Quot schemes of the projective line}

We end by giving some precise description about the Quot scheme of the projective line, and by calculating some examples. Related, but different is the work \cite{StrommeCurves}.


\begin{prop}\label{two is enough} Let $A$ be a Noetherian ring, and $V$ a projective $A$-module of rank 2. Let $F=\bigoplus_{i=1}^p\Sym_A(V)$, and let $R\subseteq F$ be a $\Sym_A(V)$-module generated in degree $d$.  Assume that $F_m/R_m$ is a projective $A$-module of rank $n\leq d$, for $m=d, d+1$. Then $F_m/R_m$ is a  projective $A$-module of rank $n$ for all $m\geq d$.
\end{prop}

\begin{proof} We may assume that $(A,\m)$ is a local ring, and that $\Sym_A(V)=A[X,Y]$. We will first show that if $E_{m}=F_m/R_m$ and $F_{m+1}/R_{m+1}$ are free of rank $n$, then $R_{m+2}$ is generated by $p(m+3)-n$ elements. Thus, we assume that $F_m/R_m$ is free of rank $n$. Then $R_m$ is free and we let $g_1, \ldots, g_t$ be a basis of $R_m$, with $t=p(m+1)-n$. It follows that the elements $Xg_1, \ldots, Xg_t$ in $R_{m+1}$ are linearly independent. Since we assume that also $F_{m+1}/R_{m+1}$ is free of rank $n$, we may assume, after possibly rearranging the order of $g_1, \ldots, g_t$, that 
\begin{equation}
\label{basis} 
Xg_1,  \ldots, Xg_t, Yg_1, \ldots, Yg_p
\end{equation}
is a basis of $R_{m+1}$. In particular we have that the remaining generators $Yg_{p+i}$ are in the linear span of the elements (\ref{basis}). In other words, 
\begin{equation}
\label{linear expression}
 Yg_{p+i}=\sum _{j=1}^t \alpha_{i,j}Xg_j +\sum_{k=1}^p \alpha_{i,k}Yg_k,
\end{equation}
for some scalars $\alpha_{i,j}$ and $\alpha_{i,k}$ in $A$, for all $i=1, \ldots, t-p$.   By definition the generators (\ref{basis}) of $R_{m+1}$ multiplied with the linear forms $X$ and $Y$ will generate $R_{m+2}$. We claim that the subset
\begin{equation}
\label{gens for n+2}
 X^2g_1, \ldots, X^2g_t, XYg_1, \ldots, XYg_p, Y^2g_1, \ldots, Y^2g_p
\end{equation}
generate $R_{m+2}$. The generators that we obtain by multiplying (\ref{basis}) with $X$ and $Y$, and that we have not listed above are the elements of the form $XYg_{p+i}$, with $i=1, \ldots, t-p$.  Using (\ref{linear expression}) we get that
$$ XYg_{p+i}=X\big( \sum _{j=1}^t \alpha_{i,j}Xg_j +\sum_{k=1}^p \alpha_{i,k}Yg_k\big),$$
and that the elements $XYg_{p+i}$ are redundant. Thus (\ref{gens for n+2}) generate $R_{m+2}$, and their cardinality is 
$$ t+2p=p(m+1)-n+2p=p(m+3)-n.$$
We then will conclude that $E_{m+2}$ is free. As we just have proved, the vector space dimension of $R_{m+2}\bigotimes_AA/\m$ is at most $p\cdot\operatorname{dim}F_{m+2} -n$, we have that the vector space dimension $\operatorname{dim}E_{m+2}\bigotimes_AA/\m \geq n$. By Proposition \ref{Gasharov} (1) it then follows that the vector space dimension of $E_{m+2}\bigotimes_A/A\m$ is exactly $n$.  Moreover, as $R_{m+2}$ is generated by 
$p\cdot \operatorname{dim}F_{m+2} -n$ elements the Fitting ideals $\Fitt_{n-s}(E_{m+2})$ are zero, for $s\geq 1$. In particular the $(n-1)$'th Fitting ideal $\Fitt_{n-1}(E_{m+2})=0$. It then follows by Proposition  \ref{Fitting} that $E_{m+2}$ is free of rank $n$. The result now follows by induction.
\end{proof}

\begin{rem} When $V$ is free of higher rank the same proof does not work, and the problem is the rearrangement done in \ref{basis}. 
\end{rem}

\begin{cor}\label{P1} Let $V$ be a projective $A$-module of rank 2, and let $R\subseteq F=\bigoplus_{i=1}^p\Sym_A(V)$ be a submodule that is generated in degree $d$. Let $E=F/R$, and  assume that the $A$-module $E_d$ is projective of rank $n$, and that $d\geq n$. Then we have equality of ideals
$$ \Fitt_{n-1}(E_{d+1})=\ufitt{n-1}{d,\infty}(E).$$
In particular we have that
$$ \mathrm{Quot}^n_{\calO^{\oplus p}/\Proj{V}/\Spec(A)}=\mathrm{Fitt}_{n-1}(\calE_{d+1}),$$
where $\calE_{d+1}$ is the $d+1$ component of the sheaf cone in Theorem \ref{main result}.
\end{cor} 
\begin{proof} By assumption $E_d$ is projective of rank $n$, so $\Fitt_{n-1}(E_d)=0$. By Proposition \ref{factorization} have that the modules $E_d$ and $E_{d+1}$ are projective when restricted to $A/\Fitt_{n-1}(E_{d+1})$. From Proposition \ref{two is enough} we then get that $E_m$ are projective $A/\Fitt_{n-1}(E_{d+1})$-modules as well, for $m\geq d$. The equality of  ideals $\Fitt_{n-1}(E_{d+1})=\ufitt{n-1}{\infty}(E_d)$ then follows from Proposition \ref{factorization}. Applying the result just proven to the sheaf cone in Theorem \ref{main result} gives the identification of the Quot scheme of the projective line.
\end{proof}

\begin{ex}\label{Hilb} {\bf The Hilbert scheme of points on the projective line.} Let $V$ be free $A$-module of rank $2$. The Quot scheme of finite, flat, locally free of rank $n$ quotients of the structure sheaf ${\calO}$ on ${\mathbb P}^1_A=\Proj{V}$ is often denoted 
$$\mathrm{Hilb}^n_{{\mathbb P}^1/\Spec(A)},$$
and referred to as the Hilbert scheme of $n$ points on the projective line. Consider the Grassmannian of locally free rank $n$ quotients of the free $A$-module of $n$-forms $S_n$ of $S=A[X,Y]=\Sym_A(V)$. Having a rank $n$ quotients is the same as having a rank one quotient of its  dual $\mathrm{Hom}_A(S_n,A)\backsimeq S_n$. We therefore have that the Grassmannian is identified with the dual projective $n$-space $G={\mathbb P}^n_A$. 

Let $\tilde{S_n}$ be the sheaf on $\Spec(A)$ corresponding to the $A$-module $S_n$. If we let $f\colon {\mathbb P}^1_G\ra G$ denote the projection map, then we have that $f_*{\calO}(n)$ corresponds to the pull-back sheaf $g^{*}\tilde{S_n}$.  On the Grassmanian $G$ we have the universal sequence
\begin{equation}\label{degree d}
 \xymatrix{0\ar[r] & \calR_n \ar[r] & f_*{\calO}(n) \ar[r] & \calE_n \ar[r] & 0.}
\end{equation}
The ${\calO}_G$-module $\calR_n$ is projective of rank one.  It is then easy to check that the image of $\calR_n\bigotimes \calO(1)$ in the sheaf $f_*{\calO}(n+1)$ of degree $n+1$ forms,  is projective of rank 2. That implies that the Fitting ideal $\calF itt_{n-1}(\calE_{n+1})$ is zero, by definition. By Corollary \ref{P1}, or using the Gotzmann Persistence Theorem, we get that the Grassmannian $G={\mathbb P}^n_A$ is the Hilbert scheme of $n$ points on the projective line.
\end{ex} 

\begin{ex}\label{Hilb2} {\bf The Hilbert scheme of hypersurfaces} There are other ways, different than the approach given in Example \ref{Hilb} to view the Hilbert scheme of points. The usual way to realize the Hilbert scheme is to consider hypersurfaces of degree $n$ in ${\mathbb P}^1$. The universal hypersurface $F(X,Y)=p_0X^n +p_1X^{n-1}Y +\cdots +p_nY^n$ has its coefficients parametrized by the projective $n$-space. One can show directly that the hypersurface given by $F(X,Y)$ in ${\mathbb P}^1\times_A{\mathbb P}^n$ is the universal family of $n$-points on the line, and we get that ${\mathbb P}^n$ is the Hilbert scheme $\mathrm{Hilb}^n_{{\mathbb P}^1/\Spec(A)}$.  This approach works for hypersurfaces in ${\mathbb P}^r$, not only the projective line, that is with $r=1$.  

If one takes the ideal sheaf generated by the kernel of the map \ref{degree d}, one gets ideal sheaf of the universal family. And this shows the connection between the usual approach and the Example \ref{Hilb}. 

For the particular case with the projective line, there exists another more geometric interpretation. One forms the $n$-fold product ${\mathbb P}^1 \times_A \cdots \times_A {\mathbb  P}^1$. The symmetric group on $n$ letters $\Sgot_d$ acts by permuting the factors, and one has that ${\mathbb P}^n_A$ is the geometric quotient. Algebraically this is obtained by taking a homogeneous version of the classical statement that the elementary symmetric functions are algebraically independent. The additional information we obtain this way is that we can consider the universal family $F(X,Y)$ described above as a product
$$ F(X,Y)= (a_1 X+b_1Y)\cdots (a_dX+b_dY)$$
of linear forms. Writing out the product gives that the elementary homogeneous invariant functions in $a_1, b_1, \cdots, a_d, b_d$ are the coefficients of the universal hypersurface $F(X,Y)$. Each linear form $a_iX+b_iY$ gives a closed subscheme of ${\mathbb  P}^1\times {\mathbb P}^1$, and corresponds therefore to a family of points in the line. The geometric way to consider the Hilbert scheme $\mathrm{Hilb}^n_{{\mathbb P}^1/\Spec(A)}$ is therefore to consider an arbitrary set of $n$ number of unordered points on the line. 
\end{ex}

\begin{ex}\label{P1xP1}{\bf Quot scheme of length one quotients.} Let $V$ be a free ${\mathbb Z}$-module of rank two. Consider the Grassmannian of locally free rank 1 quotients of the free module $V\bigoplus V$ of rank four. We have that this Grassmannian is the projective three space which we identify with $G=\operatorname{Proj}({\mathbb Z}[a,b,c,d])$. Let $f\colon {\mathbb P}^1\times G\ra G$ denote the projection map, and let $\calF$ denote the free rank 2 sheaf on ${\mathbb P}^1$. On the Grassmann scheme $G$ we have the universal sequence
\begin{equation}\label{sequence p3}
 \xymatrix{0\ar[r] &\calR_1 \ar[r] & f_*\calF(1) =\calV \bigoplus \calV \ar[r]& f_*\calO (1)=\calE_1\ar[r] & 0},
\end{equation}
where $\calV$ is the free ${\calO}_G$-module of rank two. In our identification of the Grassmannian $G$, we have that the global sections of $\calO(1)$ are $a, b, c $ and $d$ and the quotient map to ${\calO}(1)$ is represented by the matrix $\begin{bmatrix} a& b & c & d \end{bmatrix}$.

 We obtain the $\Sym_G(\calV)$-graded submodule $\calR\subseteq \Sym_G(\calV)\bigoplus \Sym_G(\calV)$ generated by $\calR_1$. Let $D_+(a)$ denote the basic open affine $\Spec(A)\subseteq G={\mathbb  P}^3$, where $ A={\mathbb Z}[\frac{b}{a},\frac{c}{a},\frac{d}{a}]$. The restriction of $\calV$ to $D_+(a)$ is given by the free $A$-module $V_A$ of rank two, and
when we restrict the sequence \ref{sequence p3} to $D_+(a)$, we get the exact sequence of $A$-modules
$$ \xymatrix{ 0\ar[r] &\bigoplus_{i=1}^3A \ar[r]^{\psi_1} & \bigoplus_{i=1}^2V_{A}\ar[r] & \calO(1)_{|A}=A\ar[r]& 0,}
$$ 
where the map $\psi_1$ is given by the matrix
$$ \psi_1 =\frac{1}{a}\begin{bmatrix} -b & -c & -d \\ a & 0 & 0 \\ 0 & a & 0 \\ 0 & 0 & a\end{bmatrix}.$$ 
If we let ${\mathbb Z}[X,Y]$ denote the coordinate ring of our projective line, and let $\{X,Y\}$ denote a basis of its one forms,  then the image of the map $\psi_1$ is generated by elements
$$ g_1=(-\frac{b}{a}X+Y,0), \quad g_2=(-\frac{c}{a}X, X) \quad \text{and}\quad g_3=(-\frac{d}{a}X, Y).$$ 
By multiplying the generators $g_1, g_2, g_3$ with the one forms $X$ and $Y$, obtain the generators of $R_2$ restricted to $D_+(a)$. If we let $\{X^2,XY,Y^2\}$ denote the basis of the quadratic forms in ${\mathbb Z}[X,Y]$, then we have the exact sequence
$$ \xymatrix{ \bigoplus_{i=1}^6 A \ar[r]^{\psi_2} & \bigoplus_{i=1}^6 \ar[r] & E_2 \ar[r] & 0.}$$
The $A$-module $E_2$ is simply the co-kernel of the map $\psi_2$, and the map $\psi_2$  is given by the matrix
$$ \psi_2 =\frac{1}{a}\begin{bmatrix} -b & 0 & -c & 0 & -d & 0\\
a & -b & 0 & -c & 0 & -d \\
0 & a & 0 & 0 & 0 & 0 \\
 0 & 0 & a & 0 & 0 & 0 \\
0 & 0 & 0&  a & a  & 0 \\
0 & 0 & 0 &0 &0 & a\end{bmatrix}.$$
 The determinant, which equals the zeroth Fitting ideal $\operatorname{Fitt}_0(E_2)$ of the quotient, one computes as $\det(\psi_2)=\frac{bc-ad}{a^2}$. Similar computations over the three remaining basic opens $D_+(b)$, $D_+(c)$ and $D_+(d)$, gives that the sought Fitting ideal is generated by $ad-bc$ in ${\mathbb Z}[a, b, c, d]$. By Corollary \ref{P1} we get that
$$ \operatorname{Proj}({\mathbb Z}[a, b, c, d]/(ad-bc))={\mathbb P}^1\times {\mathbb P}^1$$ 
is the Quot scheme
$$ \operatorname{Quot}^1_{\calO\bigoplus \calO /{\mathbb P}^1/\Spec({\mathbb Z})}.$$
\end{ex}

\begin{ex}\label{rfoldproduct} {\bf Quot scheme of length one quotients II.} Consider now the Grassmannian of rank one quotients of the free ${\mathbb Z}$-module $\bigoplus_{i=1}^pV$, where $V$ is the free ${\mathbb Z}$-module of rank two. The Grassmannian is  the projective $2p-1$ space ${\mathbb P}^{2p-1}$.

Let $\calO^{\oplus p}$ denote the trivial rank $p$ bundle on the projective line $\Proj{V}={\mathbb P}^1$ and let $f
\colon {\mathbb P}^1\times {\mathbb P}^{2p-1} \ra {\mathbb P}^{2p-1}$ denote the structure map.  We have that the Quot scheme $\mathrm{Quot}^1_{\calO^{\oplus p} /\Proj{V} /{\mathbb Z}}$ is given as a closed subscheme in $\mathrm{Proj}({\mathbb Z}[X_1, Y_1, \ldots, X_p, Y_p])={\mathbb P}^{2p-1}$. Based on the computations in Example \ref{P1xP1} we get that the quadratic equations $X_iY_j-X_jY_i$ will vanish over the closed locus defining the Quot scheme, for any pair $1\leq i<j\leq p$. We therefore have an inclusion of ideals
\begin{equation} \label{inclusion}
(X_iY_j-X_jY_i)_{1\leq i < j\leq p} \subseteq \Fitt_{0}(\calE_2)
\end{equation}
where $\Fitt _0(\calE_2)$ is the ideal that defines the Quot scheme as a closed subscheme of ${\mathbb P}^{2p-1}$ (Corollary \ref{P1}). The  quadratic equations $X_iY_j-X_jY_i$ describes the Segree embedding of ${\mathbb P}^{p-1}\times {\mathbb P}^1$ into projective $2p-1$ space. If the inclusion of ideals \ref{inclusion} was not an equality, then the Quot scheme would be a proper subscheme of ${\mathbb P}^{p-1}\times {\mathbb P}^1$. However, the product is an integral scheme, and any proper subscheme would have dimension strictly less than $p$. It is an easy fact to check that the Quot scheme is of (relative) dimension $p$, hence
$$ {\mathbb P}^{p-1} \times {\mathbb P}^1 =\mathrm{Quot}^1_{\calO^{\oplus p} / {\mathbb P}^1 / {\mathbb Z}}.$$
\end{ex}

\begin{rem}\label{Kleiman} We have that the Quot scheme $\mathrm{Quot}^1_{\calF' /X/S}=\Proj{\calF'}$ (see \cite{KleimanSitges}). Now, if $\calF'=f^*{\calF}$ for some coherent sheaf on $S$, then we have that $\Proj{f^*{\calF}}=\Proj{\calF}\times_SX$. With $\calF$ the free sheaf on $S$ and with $X={\mathbb P}^1$ we recover the description given in the examples above. We thank S. Kleiman for pointing this out.
\end{rem}

\begin{rem} From Example \ref{rfoldproduct} we obtain two defining descriptions of the same space: Let $V_r$ be the free rank $r+1$  sheaf on $\Spec(\mathbb{Z})$. Recall that if $f\colon S \ra \Spec(\mathbb{Z})$ is a scheme, then a morphism $S \ra {\mathbb P}^{r}$ to the projective line is the same as having a locally free rank one quotient sheaf
$$ \varphi \colon f^*V_r \ra \call $$
on $S$.  An $S$-valued point of the product ${\mathbb P}^{p-1} \times {\mathbb P}^1$ is therefore the same as having an ordered pair $(\varphi_1, \varphi_2)$ on line bundle quotients. As we saw in the Example \ref{rfoldproduct}, the product of $\mathbb{P}^{p-1}\times \mathbb{P}^1$ equals the Quot scheme $\mathrm{Quot}^1_{\calO^{\oplus p}/\mathbb{P}^1/{\Spec(\mathbb{Z})}}$. Consequently the ordered sequence $(\varphi_1, \varphi_2)$ is in a natural way the same as a coherent quotient of $\mathbb{P}\times S$-sheaves
$$ \xymatrix{ \bigoplus_{i=1}^p {\calO}_{\mathbb{P}^1\times S} \ra \calE},$$
where $\calE$ is flat and of relative rank one over $S$. It is unclear how to describe that natural correspondence.
\end{rem}

\begin{rem} It is unclear if the Quot functor of finite, locally free quotients of the trivial bundle of infinite rank is representable. That is, what happens when we let the rank $p$ of the free sheaf grow (see \cite{diBrino}).
\end{rem}

\begin{ex}\label{p2}{\bf Projective plane.} We end by giving the result concerning examples of Quot schemes of sheaves on spaces other than the projective line. Consider the Grassmannian ${\mathbb P}^5$ of rank one quotients of $V\bigoplus V$, where $V$ is the free ${\mathbb Z}$-module of one forms on the projective plane $V=H^0({\mathbb P}^2,{\calO}(1))$. We let $X,Y,Z$ be a basis of $V$. Inside the Grassmannian ${\mathbb P}^5$ we have the Quot scheme $Q=\mathrm{Quot}^1_{{\calO}^{\oplus 2}, {\mathbb P}^2}$, given as the intersection of a finite number of Fitting ideal strata. We will look at a certain standard open chart of the Grassmannian, namely the affine 5-space ${\mathbb A}^5$ where we have inverted the first coordinate. Let $A={\mathbb Z}[a_1, \ldots, a_5]$ be the polynomial ring. The universal family restricted to $\Spec(A)$ is then given by the matrix
$$ \begin{bmatrix} 1 & a_1 & a_2 & a_3 & a_4 &a_5\end{bmatrix}, $$
where the matrix represent the quotient map 
$$\xymatrix{V\bigoplus V\bigotimes_{\mathbb Z}A\ra A={\calO}(1)_{|\Spec(A)}}.$$
 The kernel will be generated by five elements, and generate a graded submodule $R\subseteq A[X,Y,Z]\bigoplus A[X,Y,Z]$. We let $E$ denote the graded quotient module, where $E_1=A$ is the universal family described above. Now, using the Macaulay2 software \cite{M2} one can check that the 0'th Fitting ideal $\Fitt_0(E_2)$ is the ideal
$$ (a_2a_4-a_1a_5, a_2a_3-a_5, a_1a_3-a_4) \subseteq A={\mathbb Z}[a_1, \ldots, a_5].$$
We then get that $ \mathrm{Fitt}_0(\calE_1)={\mathbb P}^2\times {\mathbb P}^1$
embedded into ${\mathbb P}^5$ via the Segree embedding. In particular we have that $\mathrm{Fitt}_0(\calE_1)$ is reduced, and it follows that
$$ \mathrm{Quot}^1_{\calO^{\oplus 2}/{\mathbb P}^2}={\mathbb P}^2\times {\mathbb P}^1.$$
In particular the equality of spaces (\ref{topos}) in fact is an equality of schemes in this case as well. Note also that the computations are aligned with the result mentioned in Remark \ref{Kleiman}.
\end{ex}


\bibliographystyle{dary}
\bibliography{paper}

\newcommand{\etalchar}[1]{$^{#1}$}
\providecommand{\bysame}{\leavevmode\hbox to3em{\hrulefill}\thinspace}
\providecommand{\MR}{\relax\ifhmode\unskip\space\fi MR }
\providecommand{\MRhref}[2]{%
  \href{http://www.ams.org/mathscinet-getitem?mr=#1}{#2}
}
\providecommand{\href}[2]{#2}
\begin{thebibliography}{BLMM11}

\bibitem[AK80]{AltmanKleiman}
Allen~B. Altman and Steven~L. Kleiman, \emph{Compactifying the {P}icard
  scheme}, Adv. in Math. \textbf{35} (1980), no.~1, 50--112.

\bibitem[{\'A}SS08]{HomogeneousHilb}
Amelia {\'A}lvarez, Fernando Sancho, and Pedro Sancho, \emph{Homogeneous
  {H}ilbert scheme}, Proc. Amer. Math. Soc. \textbf{136} (2008), no.~3,
  781--790 (electronic).

\bibitem[BH93]{BH}
Winfried Bruns and J{\"u}rgen Herzog, \emph{Cohen-{M}acaulay rings}, Cambridge
  Studies in Advanced Mathematics, vol.~39, Cambridge University Press,
  Cambridge, 1993.

\bibitem[BLMM11]{Extensors}
Jerome Brachat, Paolo Lella, Bernard Mourrain, and Roggero Margherita,
  \emph{Extensors and the {H}ilbert scheme}, {arXiv:math/1104.2007}.

\bibitem[Bou03]{Bourbaki4-7}
Nicolas Bourbaki, \emph{Algebra {II}. {C}hapters 4--7}, Elements of Mathematics
  (Berlin), Springer-Verlag, Berlin, 2003, Translated from the 1981 French
  edition by P. M. Cohn and J. Howie, Reprint of the 1990 English edition
  [Springer, Berlin; MR1080964 (91h:00003)].

\bibitem[Bri13]{diBrino}
Gennaro~Di Brino, \emph{The quot functor of a quasi-coherent sheaf},
  {arXiv:math/1212.4544v2}.

\bibitem[CFK01]{C-F&Kapranov}
Ionu{\c{t}} Ciocan-Fontanine and Mikhail Kapranov, \emph{Derived {Q}uot
  schemes}, Ann. Sci. \'Ecole Norm. Sup. (4) \textbf{34} (2001), no.~3,
  403--440.

\bibitem[FGI{\etalchar{+}}05]{FGAexplained}
Barbara Fantechi, Lothar G{\"o}ttsche, Luc Illusie, Steven~L. Kleiman, Nitin
  Nitsure, and Angelo Vistoli, \emph{Fundamental algebraic geometry},
  Mathematical Surveys and Monographs, vol. 123, American Mathematical Society,
  Providence, RI, 2005, Grothendieck's FGA explained.

\bibitem[Gas97]{Gasharov}
Vesselin Gasharov, \emph{Extremal properties of {H}ilbert functions}, Illinois
  J. Math. \textbf{41} (1997), no.~4, 612--629.

\bibitem[GD71]{EGA1}
A.~Grothendieck and J.~Dieudonn\'e, \emph{\'{E}l\'ements de g\'eom\'etrie
  alg\'ebrique. {I}. {L}e langage des sch\'emas}, vol. 166, Springer Verlag,
  Berlin, 1971, 2nd ed.

\bibitem[Got78]{Gotzmann}
Gerd Gotzmann, \emph{Eine {B}edingung f\"ur die {F}lachheit und das
  {H}ilbertpolynom eines graduierten {R}inges}, Math. Z. \textbf{158} (1978),
  no.~1, 61--70.

\bibitem[Gro63]{EGA3II}
A.~Grothendieck, \emph{\'{E}l\'ements de g\'eom\'etrie alg\'ebrique. {III}.
  \'{E}tude cohomologique des faisceaux coh\'erents. {II}}, Inst. Hautes
  \'Etudes Sci. Publ. Math. (1963), no.~17, 91.

\bibitem[Gro95]{FGAIV}
Alexander Grothendieck, \emph{Techniques de construction et th\'eor\`emes
  d'existence en g\'eom\'etrie alg\'ebrique. {IV}. {L}es sch\'emas de
  {H}ilbert}, S\'eminaire {B}ourbaki, {V}ol.\ 6, Soc. Math. France, Paris,
  1995, pp.~Exp.\ No.\ 221, 249--276.

\bibitem[GS]{M2}
Daniel~R. Grayson and Michael~E. Stillman, \emph{Macaulay2, a software system
  for research in algebraic geometry}, Available at
  http://www.math.uiuc.edu/Macaulay2/.

\bibitem[Har77]{H52}
Robin Hartshorne, \emph{Algebraic geometry}, Springer-Verlag, New York, 1977,
  Graduate Texts in Mathematics, No. 52.

\bibitem[HS04]{HaimanSturmfels}
Mark Haiman and Bernd Sturmfels, \emph{Multigraded {H}ilbert schemes}, J.
  Algebraic Geom. \textbf{13} (2004), no.~4, 725--769.

\bibitem[IK99]{IarrobinoKleiman}
Anthony Iarrobino and Vassil Kanev, \emph{Power sums, {G}orenstein algebras,
  and determinantal loci}, Lecture Notes in Mathematics, vol. 1721,
  Springer-Verlag, Berlin, 1999, Appendix C by Iarrobino and Steven L. Kleiman.

\bibitem[Kle90]{KleimanSitges}
Steven~L. Kleiman, \emph{Multiple-point formulas. {II}. {T}he {H}ilbert
  scheme}, Enumerative geometry ({S}itges, 1987), Lecture Notes in Math., vol.
  1436, Springer, Berlin, 1990, pp.~101--138.

\bibitem[Mac]{Macaulay}
F.~S. MacAulay, \emph{Some {P}roperties of {E}numeration in the {T}heory of
  {M}odular {S}ystems}, Proc. London Math. Soc. \textbf{S2-26}, no.~1, 531.

\bibitem[Mum66]{Mumfordlectures}
David Mumford, \emph{Lectures on curves on an algebraic surface}, With a
  section by G. M. Bergman. Annals of Mathematics Studies, No. 59, Princeton
  University Press, Princeton, N.J., 1966.

\bibitem[Mum88]{Redbook}
\bysame, \emph{The red book of varieties and schemes}, Lecture Notes in
  Mathematics, vol. 1358, Springer-Verlag, Berlin, 1988.

\bibitem[Nor76]{Northcott}
D.~G. Northcott, \emph{Finite free resolutions}, Cambridge University Press,
  Cambridge, 1976, Cambridge Tracts in Mathematics, No. 71.

\bibitem[Str87]{StrommeCurves}
Stein~Arild Str{\o}mme, \emph{On parametrized rational curves in {G}rassmann
  varieties}, Space curves ({R}occa di {P}apa, 1985), Lecture Notes in Math.,
  vol. 1266, Springer, Berlin, 1987, pp.~251--272.

\bibitem[Str96]{StrommeBanach}
\bysame, \emph{Elementary introduction to representable functors and {H}ilbert
  schemes}, Parameter spaces ({W}arsaw, 1994), Banach Center Publ., vol.~36,
  Polish Acad. Sci., Warsaw, 1996, pp.~179--198.

\end{thebibliography}
\end{document}